\documentclass[11pt]{amsart}
\usepackage[margin=1.35in]{geometry}
\usepackage{amscd,amsmath,amsxtra,amsthm,amssymb,stmaryrd,xr,mathrsfs,mathtools,enumerate,commath, comment, enumitem}
\usepackage{stmaryrd}
\usepackage{xcolor}
\usepackage{commath}
\usepackage{comment}
\usepackage{tikz-cd}
\usepackage{longtable} 
\usepackage{pdflscape} 
\usepackage{booktabs}
\usepackage{hyperref}
\definecolor{vegasgold}{rgb}{0.77, 0.7, 0.35}
\definecolor{darkgoldenrod}{rgb}{0.72, 0.53, 0.04}
\definecolor{gold(metallic)}{rgb}{0.83, 0.69, 0.22}
\hypersetup{
 colorlinks=true,
 linkcolor=darkgoldenrod,
 filecolor=brown,      
 urlcolor=blue,
 citecolor=darkgoldenrod,
 pdftitle={Arithmetic statistics for the Fine Selmer group in Iwasawa theory},
 }

\usepackage[all,cmtip]{xy}

\DeclareFontFamily{U}{wncy}{}
\DeclareFontShape{U}{wncy}{m}{n}{<->wncyr10}{}
\DeclareSymbolFont{mcy}{U}{wncy}{m}{n}
\DeclareMathSymbol{\Sh}{\mathord}{mcy}{"58}
\usepackage[T2A,T1]{fontenc}
\usepackage[OT2,T1]{fontenc}
\newcommand{\Zhe}{\mbox{\usefont{T2A}{\rmdefault}{m}{n}\CYRZH}}
\newtheorem{theorem}{Theorem}[section]
\newtheorem{lemma}[theorem]{Lemma}
\newtheorem{question}[theorem]{Question}
\newtheorem{conjecture}[theorem]{Conjecture}

\newtheorem{proposition}[theorem]{Proposition}
\newtheorem{corollary}[theorem]{Corollary}
\newtheorem{definition}[theorem]{Definition}

\numberwithin{equation}{section}

\theoremstyle{remark}
\newtheorem{remark}[theorem]{Remark}
\newtheorem{example}[theorem]{Example}

\newcommand{\Rfine}[1]{\mathcal{R}_{p^\infty}(E/#1)}

\newcommand{\Z}{\mathbb{Z}}

\newcommand{\Q}{\mathbb{Q}}
\newcommand{\F}{\mathbb{F}}

\newcommand{\mufine}{\mu_p^{\op{fn}}(E/\Q_{\op{cyc}})}
\newcommand{\lafine}{\lambda_p^{\op{fn}}(E/\Q_{\op{cyc}})}

\newcommand{\op}[1]{\operatorname{#1}}

\newcommand\mtx[4] { \left( {\begin{array}{cc}
 #1 & #2 \\
 #3 & #4 \\
 \end{array} } \right)}

\begin{document}
\title[Arithmetic statistics for the Fine Selmer group]{Arithmetic statistics for the Fine Selmer group in Iwasawa theory}

\author[A.~Ray]{Anwesh Ray}
\address[A.~Ray]{Centre de recherches mathématiques,
Université de Montréal,
Pavillon André-Aisenstadt,
2920 Chemin de la tour,
Montréal (Québec) H3T 1J4, Canada}
\email{anwesh.ray@umontreal.ca}

\author[R.~Sujatha]{R.Sujatha}
\address[R.~Sujatha]{Department of Mathematics\\
University of British Columbia\\
Vancouver BC, Canada V6T 1Z2}
\email[Sujatha]{sujatha@math.ubc.ca}

\begin{abstract}
We study arithmetic statistics for Iwasawa invariants for fine Selmer groups associated to elliptic curves.
\end{abstract}

\maketitle

\section{Introduction}
\label{section:intro}
\par Iwasawa studied growth questions of class groups in certain infinite towers of number fields. These initial developments led to deep questions, many of which are still unresolved. Given a number field $F$ and a prime number $p$, let $F_{\op{cyc}}$ denote the cyclotomic $\Z_p$-extension of $F$. Letting $F_n\subset F_{\op{cyc}}$ be such that $[F_n:F]=p^n$, denote by $\op{Cl}_p(F_n)$ the $p$-Sylow subgroup of the class group of $F_n$. Iwasawa showed that for $n\gg 0$,
\begin{equation}\label{first equation}\# \op{Cl}_p(F_n)=p^{p^n\mu+\lambda n+\nu},\end{equation} where $\mu, \lambda\in \Z_{\geq 0}$ and $\nu\in \Z$. Let $K(F_{\op{cyc}})\subset \bar{K}$ be the maximal abelian pro-$p$ extension of $F_{\op{cyc}}$ in which all primes of $F_{\op{cyc}}$ are unramified. The classical $\mu=0$ conjecture of Iwasawa states that $X(F_{\op{cyc}}):=\op{Gal}(K(F_{\op{cyc}})/F_{\op{cyc}})$ is finitely generated as a $\Z_p$-module. Equivalently, the invariant $\mu$ in \eqref{first equation} is equal to $0$. The $\mu=0$ conjecture was proved by Ferrero and Washington \cite{ferrero1979iwasawa} for all abelian number fields $F/\Q$. 
\par Mazur initiated the Iwasawa theory of elliptic curves and abelian varieties in \cite{mazur72}. Let $E$ be an elliptic curve and $p$ an odd prime number at which $E$ has good ordinary reduction. The main object of study is the $p$-primary Selmer group over the cyclotomic $\Z_p$-extension of $\Q$. Greenberg conjectured that when the residual representation on the $p$-torsion subgroup of $E$ is irreducible, the Iwasawa $\mu$-invariant of this Selmer group must vanish, see \cite[Conjecture 1.11]{Gre99}. This means that the Pontryagin dual of the Selmer group over $\Q_{\op{cyc}}$ is a finitely generated $\Z_p$-module. The rank of the dual Selmer group as a $\Z_p$-module is the $\lambda$-invariant, and is of significant interest. It was proven by Kato \cite{kato2004p} and Rohlrich \cite{rohrlich1984onl} that the rank of $E(\Q_n)$ is bounded as $n\rightarrow \infty$. In fact, $\op{rank} E(\Q_n)$ is always bounded above by the $\lambda$-invariant. When the residual representation is reducible, there are examples and explicit criteria under which the $\mu$-invariant is positive, see \cite[sections 3 and 7]{ray2021mu}. 
\par Let $E_{/\Q}$ be an elliptic curve with good reduction at an odd prime $p$ (either ordinary or supersingular). The \emph{fine Selmer group} was systematically studied by Coates and the second named author in \cite{coates2005fine}. This is the subgroup $\Rfine{\Q_{\op{cyc}}}$ of the $p$-primary Selmer group over $\Q_{\op{cyc}}$ consisting of all cohomology classes that are trivial when localized at the primes above $p$ and the primes at which $E$ has bad reduction (cf. Definitions \ref{fine Selmer def} and \ref{fine Selmer def 2}). The fine Selmer group is closely related to class group extensions studied by Iwasawa. Conjecture A in \cite{coates2005fine} predicts that the $\mu$-invariant of the fine Selmer group is always zero. Let $F=\Q(E[p])$ be the extension of $\Q$ which is \emph{cut out} by the residual representation on $E[p]$. In other words, it is the Galois extension of $\Q$ fixed by the kernel of the residual representation. By Serre's Open image theorem \cite[section 4, Theorem 3]{Serre72}, if $E$ is a non-CM elliptic curve, $\op{Gal}(F/\Q)$ is isomorphic to $\op{GL}_2(\Z/p\Z)$ for all but finitely many primes $p$. Thus, the extension $F/\Q$ is in most cases non-abelian and in fact, non-solvable; and Iwasawa's $\mu=0$ conjecture is wide open for such extensions. It follows from results of Coates and the second named author that Iwasawa's $\mu=0$ conjecture for $\op{Gal}(K(F_{\op{cyc}})/F_{\op{cyc}})$ is closely related to the $\mu=0$ conjecture for the fine Selmer group. In greater detail, the $\mu$-invariant vanishes for the fine Selmer group over $F_{\op{cyc}}$ vanishes if and only if the classical Iwasawa $\mu$-invariant for $X(F_{\op{cyc}})$ vanishes (cf. Theorem \ref{class group relation to conjecture A}).  

\par The arithmetic statistics of elliptic curves is concerned with the study of the average behaviour of certain interesting arithmetic invariants associated with elliptic curves. One such invariant is the rank of the Mordell-Weil group. A conjecture of Katz and Sarnak \cite{katz1999random} supported with heuristics from random matrix theory predicts that when ordered by height (or conductor or discriminant), $50\%$ of all elliptic curves have rank $0$ and the other $50\%$ have rank $1$. Certain unconditional results on rank distribution have been proven by Bhargava and Shankar \cite{bhargava2013average} who study the average cardinality of certain Selmer groups over $\Q$. This provides ample motivation to study the average behavior of Iwasawa invariants of elliptic curves associated to Selmer groups over $\Q_{\op{cyc}}$. The main difficulty is that the methods of Bhargava and Shankar can only be applied to $n$ Selmer groups over $\Q$ and for numbers $n\leq 5$. Therefore, an altogether different method is required to study Selmer groups over the cyclotomic $\Z_p$-extension. Arithmetic statistics for the Iwasawa theory of elliptic curves was initiated by Kundu and the first named author in \cite{KR21}. In the present article, we study similar questions for the fine Selmer group. In greater detail, we study the following questions:
\begin{enumerate}
    \item Fix a prime number $p$. What can be said about the proportion of elliptic curves $E_{/\Q}$ with good reduction at $p$ for which the $p$-primary fine Selmer group over $\Q_{\op{cyc}}$ is infinite (i.e., either $\mu>0$ or $\lambda>0$)? Here, the elliptic curves are ordered according to their \emph{naive height} (cf. section \ref{s 4} for further details).
    \item Given an elliptic curve $E_{/\Q}$ and a fixed positive number $Y>0$, what can be said about the number of primes $p\leq Y$ at which $E$ has good reduction and the fine Selmer group of $E$ over $\Q_{\op{cyc}}$ is infinite.
\end{enumerate}
In section \ref{s 4}, we study the first question and prove explicit results. From a statistical point of view it makes sense to only consider elliptic curves of rank $\leq 1$ (since the elliptic curves of rank $>1$ are expected to constitute a set of density $0$). Let $p$ be a prime $\geq 5$. Consider the set $\mathfrak{B}_p$ (resp. $\mathfrak{D}_p$) consisting of elliptic curves $E_{/\Q}$ that have Mordell-Weil rank $0$ (resp. $1$), good reduction at $p$, and for which the $p$-primary fine Selmer group $\Rfine{\Q_{\op{cyc}}}$ is infinite. Theorem \ref{th 4.6} gives a precise upper bound for the upper natural density of $\mathfrak{B}_p$ in terms of $p$. In greater detail, let $\bar{\mathfrak{d}}(\mathfrak{B}_p)$ denote the upper density of $\mathfrak{B}_p$ (see Definition \ref{density def} for a precise definition). Let $f,g$ and $h$ be functions defined on the set of primes $p$, taking positive real values. Then, we say that $f(p)\leq g(p)+O(h(p))$ if $f(p)\leq g(p) + c h(p)$ for some constant $c\geq 0$. We view the association $p\mapsto \bar{\mathfrak{d}}(\mathfrak{B}_p)$ as a function on the set of prime numbers $p$. We show that \[\bar{\mathfrak{d}}(\mathfrak{B}_p)\leq \frac{2\zeta(10)+1}{p}+O\left(\frac{\op{log}p(\op{log log} p)^2}{p^{3/2}}\right).\] A similar bound is also obtained for the function $p\mapsto \bar{\mathfrak{d}}(\mathfrak{D}_p)$, see Theorem \ref{th 4.8}. The bounds we obtain for $\bar{\mathfrak{d}}(\mathfrak{B}_p)$ and $\bar{\mathfrak{d}}(\mathfrak{D}_p)$ are conditional. For all primes $p$ for which our bounds are valid, it is assumed that $\Sh(E/\Q)[p^\infty]$ is finite. Furthermore, we assume a conjecture of Delaunay on the density of elliptic curves $E_{/\Q}$ for which $\Sh(E/\Q)[p^\infty]\neq 0$. We refer to Conjecture \ref{Del} for a precise statement. One also considers the set $\mathfrak{F}_p$ of all elliptic curves $E_{/\Q}$ with Mordell-Weil rank $0$ and good \emph{ordinary} reduction at $p$, such that the entire Selmer group $\op{Sel}_{p^\infty}(E/\Q_{\op{cyc}})$ is infinite. The bound obtained for $\bar{\mathfrak{d}}(\mathfrak{F}_p)$ in \cite[section 4]{KR21} is of the form 
\[\bar{\mathfrak{d}}(\mathfrak{F}_p)=O\left(\frac{\op{log}p(\op{log log} p)^2}{p^{1/2}}\right)\] (see Theorem \ref{kundu ray theorem} for the refined statement). The results show that that given a prime $p$ the fine Selmer group $\Rfine{\Q_{\op{cyc}}}$ of most elliptic curves is finite. Furthermore, we give an explicit upper bound for the proportion of elliptic curves for which $\Rfine{\Q_{\op{cyc}}}$ may be infinite and this proportion becomes smaller as $p$ gets larger. In fact, the quantity decreases at the rate of $O(1/p)$. On the other hand, for elliptic curves of rank 0, the Selmer group $\op{Sel}_{p^\infty}(E/\Q_{\op{cyc}})$ shows similar behavior on average. However, as $p$ gets larger, the upper bound decreases at a much slower rate. 
\par In section \ref{s 5}, we study a variant of the second question. Let $E$ be an elliptic curve with Mordell-Weil rank zero. According to Corollary \ref{ cor 2.9}, the set of primes $p$ at which $\Rfine{\Q_{\op{cyc}}}$ is infinite is contained in the set of primes $\Pi_{lt,E}\cup \mathfrak{S}$. Here, $\Pi_{lt, E}$ is the set of primes $p$ at which $E(\Q_p)[p]\neq 0$. Such primes are referred to as the set of \emph{local torsion primes}. On the other hand, $\mathfrak{S}$ is an explicit finite set of primes (see \emph{loc. cit.} for details). Given a positive number $Y>0$, we denote by $\Pi_{lt,E}(Y)$ the set of local torsion primes $p$ that are $\leq Y$. We prove results about the expected size of $\# \Pi_{lt, E}(Y)$ when $Y$ is fixed and $E$ ranges over all elliptic curves. This gives us some insight into the number of primes $\leq Y$ at which the fine Selmer group may be infinite. We use a version of the Large Sieve inequality due to Huxley, see Theorem \ref{huxleythm}. Fix $Y>0$ (for example $Y=100$), let $P(Y):=\sum_{p\leq Y} \frac{\# \mathfrak{A}_p}{p^4}$ be the quantity defined in the discussion following the proof of Theorem \ref{mean square}. Here, the fraction $\frac{\# \mathfrak{A}_p}{p^4}$ represents the proportion of elliptic curves $\mathcal{E}$ over $\Z_p$ such that $\mathcal{E}(\Q_p)[p]\neq 0$. We refer to Definition \ref{ definition Cp Sp} for the precise definition of the set $\mathfrak{A}_p$. The main result of section \ref{s 5}, Theorem \ref{th s5 main}, shows that $P(Y)$ is a good estimate for $\#\Pi_{lt,E}(Y)$. More precisely, it is shown that there is an absolute constant $c>0$ such that for any number $\beta>0$, the proportion of all elliptic curves $E_{/\Q}$ for which \begin{equation}\label{above inequality}|\#\Pi_{lt,E}(Y)-P(Y)|<\beta \sqrt{P(Y)}\end{equation} is $\geq (1-\frac{c}{\beta^2})$. Thus, the average size of $\# \Pi_{lt, E}(Y)$ is close to $P(Y)$ for an explicit positive density set of elliptic curves. For instance, setting $\beta=10\sqrt{c}$, we find that $\geq 99\%$ of elliptic curves satisfy the above inequality \eqref{above inequality}. These results shed light on the distribution of primes outside of which the fine Selmer group is finite, when the elliptic curve $E$ in question has Mordell-Weil rank zero.

\subsection*{Acknowledgements} The first named author's research is supported by the CRM-Simons postdoctoral fellowship. He would like to thank Jeffrey Hatley, Debanjana Kundu and Antonio Lei for helpful discussions and comments. He would also like to thank Christian Wuthrich and Will Sawin for their suggestions in response to a question posted by the first named author on Mathoverflow.org. The second named author gratefully acknowledges support from NSERC Discovery grant 2019-03987. The authors would like to thank the referees for helpful suggestions.
\section{Iwasawa theory of the Fine Selmer group}
\par Let $E$ be an elliptic curve defined over $\Q$ and $p$ an odd prime at which $E$ has good reduction. Fix an algebraic closure $\bar{\Q}$ of $\Q$ and a finite set of primes $S$ containing $\{p,\infty\}$ outside of which $E$ has good reduction. Let $\Q_S$ be the maximal extension of $\Q$ in $\bar{\Q}$ in which all primes $\ell\notin S$ are unramified. Denote by $E[p^n]$ the $p^n$-torsion subgroup of $E(\bar{\Q})$, and set $E[p^\infty]$ to be the the union of all $p$-primary torsion groups $E[p^n]$. Given a field $L\subset \Q_S$, set $H^i(\Q_S/L, \cdot):=H^i(\op{Gal}(\Q_S/L), \cdot)$. For a number field $L$ and a prime $\ell\in S$, set $K_\ell(\cdot /L)$ to be the direct sum $\bigoplus_{v|\ell} H^1(L_v, \cdot)$. Here, $v$ ranges over the primes of $L$ above $\ell$.

\begin{definition}\label{fine Selmer def}Let $L$ be a number field. The \emph{fine Selmer group} $\Rfine{L}$ is defined to be the kernel of the restriction map 
\[\op{ker}\left\{H^1\left(\Q_S/L, E[p^\infty]\right)\longrightarrow \bigoplus_{\ell \in S} K_\ell(E[p^\infty]/L)\right\}.\]
\end{definition}
\par We study the fine Selmer group from an Iwasawa theoretic point of view. This involves passing up an infinite tower of number fields. Let $\Q(\mu_{p^\infty})$ be the field generated by $\Q$ and the $p$-power roots of unity $\mu_{p^\infty}\subset \bar{\Q}$. For $n\in \Z_{\geq 1}$, let $\Q_n\subset \Q(\mu_{p^\infty})$ be the extension of $\Q$ contained in $\Q(\mu_{p^\infty})$ such that $[\Q_n:\Q]=p^n$. The \emph{cyclotomic $\Z_p$-extension} of $\Q$ is the union $\Q_{\op{cyc}}:=\bigcup_{n\geq 1} \Q_n$, with Galois group $\Gamma:=\op{Gal}(\Q_{\op{cyc}}/\Q)$. Fix a topological generator $\gamma\in \Gamma$. This gives rise to an isomorphism $\Z_p\xrightarrow{\sim} \Gamma$, sending $a\in \Z_p$ to $\gamma^a$. The \emph{Iwasawa algebra} is the completed group algebra $\Lambda:=\varprojlim_n \Z_p[\op{Gal}(\Q_n/\Q)]$. Fix an isomorphism of $\Lambda$ with the formal power series ring $\Z_p\llbracket T \rrbracket$, where $T$ is the formal variable in place of $(\gamma-1)\in \Lambda$.

\par \begin{definition}\label{fine Selmer def 2}The fine Selmer group over $\Q_{\op{cyc}}$ is taken to be the direct limit 
\begin{equation}\label{fine selmer def 2}\Rfine{\Q_{\op{cyc}}}:=\varinjlim_n \Rfine{\Q_{n}}\end{equation} and is a cofinitely generated module over $\Lambda$.
\end{definition}In other words, the Pontryagin dual 
\[\mathcal{R}_{p^\infty}(E/\Q_{\op{cyc}})^\vee:=\op{Hom}_{\op{cts}}\left(\mathcal{R}_{p^\infty}(E/\Q_{\op{cyc}}), \Q_p/\Z_p\right)\] is finitely generated over $\Lambda$. We refer to \cite{coates2000galois} for the definition of the $p$-primary Selmer group over $\Q_{\op{cyc}}$, which we denote by $\op{Sel}_{p^\infty}(E/\Q_{\op{cyc}})$. Note that the fine Selmer group is a subgroup of $\op{Sel}_{p^\infty}(E/\Q_{\op{cyc}})$. It follows from a deep result of Kato \cite{kato2004p} that $\op{Sel}_{p^\infty}(E/\Q_{\op{cyc}})^\vee$ is torsion over $\Lambda$. In particular it follows that $\Rfine{\Q_{\op{cyc}}}^\vee$ is torsion over $\Lambda$.
\par Let $M$ be a cofinitely generated and cotorsion $\Lambda$-module. We introduce the Iwasawa invariants associated to the Pontryagin dual $M^\vee:=\op{Hom}\left(M, \Q_p/\Z_p\right)$. Associated with $M^\vee$ are its Iwasawa invariants. By the \emph{Structure Theorem} for finitely generated and torsion $\Lambda$-modules \cite[Theorem 13.12]{washington1997}, $M^{\vee}$ is pseudo-isomorphic to a finite direct sum of cyclic $\Lambda$-modules. In other words, there is a map of $\Lambda$-modules
\[
M^{\vee}\longrightarrow \left(\bigoplus_{i=1}^s \Lambda/(p^{\mu_i})\right)\oplus \left(\bigoplus_{j=1}^t \Lambda/(f_j(T)) \right)
\]
with finite kernel and cokernel. Here, $\mu_i>0$ and $f_j(T)$ is a distinguished polynomial (i.e., a monic polynomial with non-leading coefficients divisible by $p$).
The characteristic ideal of $M^\vee$ is (up to a unit) generated by
\[
 f_{M}(T) := p^{\sum_{i} \mu_i} \prod_j f_j(T)=p^\mu g_M(T),
\]
where $\mu=\sum_i \mu_i$ and $g_M(T)$ is the distinguished polynomial $\prod_j f_j(T)$. The quantity $\mu$ is the $\mu$-invariant, which is denoted $\mu_p(M)$, and the $\lambda$-invariant $\lambda_p(M)$ is the degree of $g_M(T)$. In this paper, we shall set $\mu_p(E/\Q_{\op{cyc}})$ (resp. $\lambda_p(E/\Q_{\op{cyc}})$) to denote the $\mu$-invariant (resp. $\lambda$-invariant) of the classical Selmer group $\op{Sel}_{p^\infty}(E/\Q_{\op{cyc}})$. On the other hand, we set $\mu_p^{\op{fn}}(E/\Q_{\op{cyc}})$ (resp. $\lambda_p^{\op{fn}}(E/\Q_{\op{cyc}})$) to denote the $\mu$-invariant (resp. $\lambda$-invariant) of the fine Selmer group $\mathcal{R}_{p^\infty}(E/\Q_{\op{cyc}})$. At this point, we note that for the fine Selmer group, there is no known analog of the main conjecture. Therefore, the fine Selmer group is a purely algebraic object and in many respects, its structure is more mysterious than the classical Selmer group.

\par Let $E$ be an elliptic curve with good reduction at $p$ and let $F:=\Q(E[p])$, i.e., the field which is fixed by the kernel of the mod-$p$ representation on $E[p]$. Let $\mu_p(F_{\op{cyc}}/F)$ be the classical Iwasawa $\mu$-invariant associated with the maximal abelian pro-$p$ extension of $F_{\op{cyc}}$ which is unramified at all primes. The classicial $\mu=0$ conjecture of Iwasawa predicts that $\mu_p(K_{\op{cyc}}/K)=0$ for any number field $K$. The following result due to Coates and the second named author establishes a relationship between the vanishing of the classical $\mu$-invariant and the $\mu$-invariant of the fine Selmer group of $E$ over $F$.

\begin{theorem}\label{class group relation to conjecture A}
Let $E$ be an elliptic curve and $p$ an odd prime number. Let $F$ be the number field given by $\Q(E[p])$. Then, the following are equivalent
\begin{enumerate}
\item $\mathcal{R}_{p^\infty}(E/F_{\op{cyc}})^\vee$ is a torsion $\Lambda$-module with $\mu_p^{\op{fn}}(E/F_{\op{cyc}})=0$,
\item $\mu_p(F_{\op{cyc}}/F)=0$.
\end{enumerate}
\end{theorem}
\begin{proof}
The statement follows directly from from \cite[Theorem 3.4]{coates2005fine}. 
\end{proof}
One should note that the above result in fact provides us with a criterion for the classical $\mu=0$ conjecture to hold for number field extensions $F=\Q(E[p])$, which need not be abelian or even solvable. However, instead of studying the vanishing of $\mu_p^{\op{fn}}(E/F_{\op{cyc}})$, we shall study conditions for the vanishing of $\mu_p^{\op{fn}}(E/\Q_{\op{cyc}})$. Conjecture A of \cite{coates2005fine} asserts that $\mufine=0$ for all elliptic curves $E_{/\Q}$. The above result shows that it is a special case of the $\mu=0$ conjecture for number fields. 

\par Now consider for a moment the special case when the residual representation
\[\bar{\rho}:\op{Gal}(\bar{\Q}/\Q)\rightarrow \op{GL}_2(\F_p)\]on $E[p]$ is reducible, i.e., $\bar{\rho}=\mtx{\varphi_1}{\ast}{0}{\varphi_2}$, for globally defined characters \[\varphi_i: \op{Gal}(\bar{\Q}/\Q)\rightarrow \F_p^\times.\]
\begin{proposition}
Let $E_{/\Q}$ be an elliptic curve and $p$ an odd prime such that $\bar{\rho}$ is reducible, then, $\mufine=0$.
\end{proposition}
\begin{proof}
Let $\Q(\varphi_i)$ be the field fixed by the kernel of $\varphi_i$ and set $L$ to be the composite of $\Q(\varphi_1)$ and $\Q(\varphi_2)$. Note that $\Q(E[p^\infty])/L$ is pro-$p$ and since $L$ an abelian number field, by the Theorem of Ferrero and Washington \cite{ferrero1979iwasawa}, the Iwasawa $\mu=0$ conjecture holds for $L_{\op{cyc}}/L$. It follows from \cite[Corollary 3.5]{coates2005fine} that in fact, $\mu_p^{\op{fn}}(E/L_{\op{cyc}})=0$, and it is easy to see that this implies that $\mufine=0$.
\end{proof}
In contrast, the $\mu$-invariant of the Selmer group $\op{Sel}_{p^\infty}(E/\Q_{\op{cyc}})$ is known to be positive in some cases when $\bar{\rho}$ is reducible. From a statistical point of view, the case of interest is that when $\bar{\rho}$ is surjective, in particular, irreducible. It follows from Serre's Open image theorem that when $E_{/\Q}$ is an elliptic curve without complex multiplication, the residual representation on $E[p]$ is surjective for all but finitely many primes $p$. A \emph{Serre curve} is an elliptic curve whose adelic Galois representation has index-$2$ in $\op{GL}_2(\hat{\Z})$, see \cite{jones2010almost} for a more precise definition. Serre curves are of interest since they are the elliptic curves for which the adelic Galois image is as large as possible. It is not hard to see that given a Serre curve, the mod-$p$ Galois representation is irreducible for all odd primes $p$. Jones proves in \emph{loc. cit.} that $100\%$ of elliptic curves ordered according to height are Serre curves.
\par Let $f(T)$ be the characteristic element of the fine Selmer group $\mathcal{R}_{p^\infty}(E/\Q_{\op{cyc}})$, and write $f(T)=a_r T^r+a_{r+1}T^{r+1}+\dots$, where $a_r\in \Z_p$ is non-zero. Here, the number $r$ is called the order of vanishing of $f(T)$ at $T=0$ and $a_r$ is called the \emph{leading term}. The class of $a_r$ modulo $\Z_p^\times$ does not depend on the choice of topological generator $\gamma\in \Gamma$. Let $\mathcal{M}(E/\Q)$ be the subgroup of elements of $\Rfine{\Q}$ that lie in the image of the Kummer map
\[E(\Q)\otimes \Q_p/\Z_p\rightarrow H^1(\Q_S/\Q, E[p^\infty]),\] and the \emph{fine Tate-Shafarevich group} $\Zhe_{p^\infty}(E/\Q)$ is the cokernel of
\[\mathcal{M}(E/\Q)\rightarrow \Rfine{\Q}.\] Then, $\Zhe_{p^\infty}(E/\Q)$ is naturally identified with a subgroup of $\Sh(E/K)[p^\infty]$. In our applications, we would like to understand when $\Zhe_{p^\infty}(E/\Q)$ is non-zero. The following result shows that the fine Tate-Shafarevich group vanishes precisely when the $p$-primary part of the Tate-Shafarevich group does.
\begin{theorem}[Wuthrich]\label{theorem 2.5}
Let $E$ be an elliptic curve defined over $\Q$. Then, $\Zhe_{p^\infty}(E/\Q)\neq 0$ precisely when $\Sh(E/\Q)[p^\infty]\neq 0$. Furthermore, if $\op{rank} E(\Q)>0$, then, $\Zhe_{p^\infty}(E/\Q)=\Sh(E/\Q)[p^\infty]$.
\end{theorem}
\begin{proof}
See \cite[Theorems 3.4, 3.5]{wuthrich2007fine}.
\end{proof}
When $\op{rank} E(\Q)=0$, $\Zhe_{p^\infty}(E/\Q)$ may be strictly smaller than $\Sh(E/\Q)[p^\infty]$, as numerical examples in \cite{wuthrich2007fine} show.
\par Wuthrich gives an explicit formula for the leading term of the characteristic element $a_r$, see \cite[Corollary 6.2]{wuthrich2007iwasawa}.  We let $D=D_{E,p}$ be the cokernel of the natural map $E(\Q)\otimes \Z_p$ to the $p$-adic completion of $E(\Q_p)$. At a prime $\ell$ at which $E$ has bad reduction, $c_\ell(E)$ will denote the Tamagawa number at $\ell$. Given two numbers $a,b$, we write $a\sim b$ if $a=ub$ for some unit $u\in \Z_p^\times$. There is a $p$-adic height pairing on the Selmer group $\Rfine{\Q}$, see \cite[section 5]{wuthrich2007iwasawa}. The regulator $\op{Reg}\left(\Rfine{\Q}\right)$ is the determinant of this height pairing. The compact version of the fine Selmer group $\mathfrak{R}_{p^\infty}(E/\Q)$ is the kernel of the localization map
\[\mathfrak{R}_{p^\infty}(E/\Q):=\op{ker}\left(H^1(\Q_S/\Q, T_p(E))\rightarrow  H^1(\Q_p, T_p(E))\right).\]
Here, $T_p(E)$ denotes the $p$-adic Tate-module associated to $E$.
\begin{theorem}[Wuthrich]\label{wuthrich thm}
Let $E$ be an elliptic curve over $\Q$ with potentially good reduction at $p$. Assume that the following conditions are satisfied:
\begin{enumerate}[label=(\alph*)]
    \item The fine Tate-Shafarevich group $\Zhe_{p^\infty}(E/\Q)$ is finite.
    \item The cyclotomic height pairing on $\Rfine{\Q_{\op{cyc}}}$ is nondegenerate.
\end{enumerate}
Then, the following assertions hold
\begin{enumerate}
\item $r=\op{max}\left(0,\op{rank} E(\Q)-1\right)$.
\item There is an injection with finite cokernel $J$ of $\mathfrak{R}_{p^\infty}(E/\Q)$ into the cokernel of the corestriction map 
\[\op{cor}:\varprojlim_n H^1(\Q_S/\Q_n, T_p(E))\rightarrow H^1(\Q_S/\Q, T_p(E)).\]
\item We have the following formula for the leading term $a_r$
\begin{equation}\label{leading term}a_r\sim \op{Reg}\left(\Rfine{\Q}\right)\times \left(\frac{\#\op{Tors}_{\Z_p}(D)\times \prod_{\ell\neq  p} c_\ell(E) \times \#\Zhe_{p^\infty}(E/\Q)}{\# J}\right),\end{equation}
where, $\op{Tors}_{\Z_p}(D)$ denotes the $p$-primary torsion subgroup of $D$.
\end{enumerate}

\end{theorem}
\begin{proof}
The above result is \cite[Theorem 1.1]{wuthrich2007fine}.
\end{proof}
Note that under the assumptions of the above theorem, $r=0$ and the regulator $\op{Reg}\left(\Rfine{\Q}\right)=1$ when $\op{rank} E(\Q)\leq 1$. As a result of the above formula, we obtain an important consequence towards the vanishing of the Iwasawa invariants of the fine Selmer group $\Rfine{\Q_{\op{cyc}}}$.

\begin{corollary}\label{criterion for mu=lambda=0}
Let $E$ be an elliptic curve over $\Q$ and $p$ an odd prime. Assume that $E$ has potentially good reduction at $p$ and satisfied the conditions of Theorem \ref{wuthrich thm}. Furthermore, assume that the following conditions hold:
\begin{enumerate}
    \item $\op{rank} E(\Q)\leq 1$,
    \item $p\nmid c_\ell(E)$ for all primes $\ell\neq p$,
    \item $\Zhe_{p^\infty} (E/\Q)=0$,
    \item $\op{Tors}_{\Z_p}(D)=0$.
\end{enumerate}
Then, the $\mu$ and $\lambda$-invariants of the fine Selmer group $\Rfine{\Q_{\op{cyc}}}$ are $0$. In particular, $\Rfine{\Q_{\op{cyc}}}$ has finite cardinality.
\end{corollary}

\begin{proof}
It follows from Theorem \ref{wuthrich thm} that $r=0$. Furthermore, the regulator \[\op{Reg}\left(\Rfine{\Q}\right)=1.\] Furthermore, the conditions imply that $a_0=1$, hence the characteristic element of $\Rfine{\Q_{\op{cyc}}}$ is a unit in $\Lambda$. The result follows from this.
\end{proof}
The above result has an interesting consequence from a statistical point of view. Let us first introduce some notation. Given $x>0$, recall that $\pi(x)$ is the number of primes $p\leq x$ and that the prime number theorem states that $\pi(x)\sim \frac{x}{\op{log}(x)}$, i.e., 
\[\lim_{x\rightarrow\infty} \frac{\pi(x)}{\left(x/\op{log}(x)\right)}=1.\]
Given a set of primes $\mathcal{S}$, set $\mathcal{S}(x):=\{p\in \mathcal{S}\mid p\leq x\}$, note that $\# \mathcal{S}(x)\leq \pi(x)$ by definition. 
\begin{definition}
Let $\beta\in [0,1]$ and $\mathcal{S}$ be a set of prime numbers. We say that $\mathcal{S}$ has density $\beta$ if
\[\lim_{x\rightarrow \infty} \frac{\# \mathcal{S}(x)}{\pi(x)}=\beta.\] Note that part of the requirement is that the limit actually exists. 
\end{definition}
Thus, if the set $\mathcal{S}$ has density $1$, we say that $\mathcal{S}$ consists of $100\%$ of primes. If the complement of $\mathcal{S}$ in the set of all primes is finite, then indeed, $\mathcal{S}$ consists of $100\%$ of primes, however, not conversely. Given an elliptic curve $E_{/\Q}$, let $\Pi_{lt,E}$ be the set of \emph{local torsion primes} for $E$, i.e., the set of all primes $p$ such that $E$ has good reduction at $p$ and $E(\Q_p)$ contains a $p$-torsion point. Let $\Pi_{an, E}$ be the set of \emph{anomalous primes}, i.e., primes $p$ at which $E$ has good reduction such that $\widetilde{E}(\F_p)[p]\neq 0$. Here, $\widetilde{E}$ is the reduction of $E$ at $p$. Let $\Pi_E$ be the set of primes $p$ at which $E$ has good reduction such that $\op{Tors}_{\Z_p} D\neq 0$.
\begin{proposition}\label{PiE}
Let $E$ be an elliptic curve over $\Q$, then, $\Pi_{lt,E}\subseteq \Pi_{an, E}$. Furthermore, if $\op{rank}E(\Q)=0$, then, $\Pi_E\subseteq \Pi_{lt, E}$.
\end{proposition}
\begin{proof}
Let $E$ have good reduction at $p$, in other words, $E$ can be viewed as an elliptic curve over $\op{Spec}\Z_p$. Letting $E_0$ be the formal group of $E$, we have a short exact sequence
\[0\rightarrow E_0(p\Z_p)\rightarrow E(\Q_p)\rightarrow \widetilde{E}(\F_p)\rightarrow 0.\] Note that $E_0(p\Z_p)\simeq \Z_p$ does not contain any $p$-torsion, and therefore, there is an injection $E(\Q_p)[p^\infty]\hookrightarrow \widetilde{E}(\F_p)[p^\infty]$. This shows that if $p$ is a local torsion prime, then $p$ is an anomalous prime, i.e., $\Pi_{lt,E}\subseteq \Pi_{an, E}$. On the other hand, when $\op{rank} E(\Q)=0$ and $E(\Q_p)$ has no nontrivial $p$-torsion, it follows that $\op{Tors}_{\Z_p} D=0$. Therefore, $\Pi_E\subseteq \Pi_{lt,E}$, when the Mordell-Weil rank of $E$ is $0$.
\end{proof}
\begin{corollary}\label{ cor 2.9}
Let $E$ be a non-CM elliptic curve over $\Q$ with $\op{rank} E(\Q)=0$ for which the Tate-Shafarevich group $\Sh(E/\Q)$ is finite. Then, for $100\%$ of primes $p$, \[\mufine=0\text{ and }\lafine=0.\] The set of primes for which $\mufine>0$ or $\lafine>0$ is contained in the set $\Pi_{lt, E}\cup \mathfrak{S}$, where $\mathfrak{S}$ is the set of primes dividing $ 2 \# \Sh(E/\Q)\times \prod_{\ell\neq p} c_\ell(E)$.
\end{corollary}

\begin{proof}
Since it is assumed that $\Sh(E/\Q)$ is finite, it follows that $\Zhe_{p^\infty}(E/\Q)=0$ for all but finitely many primes. It is also clear (from the fact that $c_\ell(E)$ is a non-zero integer) that $\prod_{\ell\neq p} c_\ell(E)$ is a unit in $\Z_p$ for all but finitely many primes $p$. It is well known that $\Pi_{\op{an}, E}$ has density zero, see \cite{murty1997modular}. Since it is assumed that $\op{rank} E(\Q)=0$, Proposition \ref{PiE} asserts that $\Pi_{E}\subseteq \Pi_{\op{an},E}$, and hence $\Pi_{E}$ has density zero. Therefore, we have shown that the conditions of Corollary \ref{criterion for mu=lambda=0} hold for $100\%$ of primes $p$.
\end{proof}

\section{Statistics for local torsion primes}
Let $E_{/\Q}$ be an elliptic curve. We recall that a prime $p$ at which $E$ has good reduction is said to be \emph{anomalous} if $p\mid \#\widetilde{E}(\F_p)$. Given an elliptic curve $E$, $\Pi_{an,E}$ is the set of anomalous primes for $E$. The following result is due to Greenberg, see \cite[Theorems 4.1 and 5.1]{Gre99}.

\begin{theorem}[Greenberg]
Let $E$ be an elliptic curve over $\Q$ without complex multiplication such that $\op{rank} E(\Q)=0$ and $\Sh(E/\Q)$ is finite. Then, for $100\%$ of primes $p$, the $p$-primary Selmer group $\op{Sel}_{p^\infty}(E/\Q_{\op{cyc}})$ has $\mu=0$ and $\lambda=0$. The set of primes $p$ at which either $\mu>0$ or $\lambda>0$ is contained in $\Pi_{an,E}\cup \mathfrak{S}$, where $\mathfrak{S}$ is the finite set of primes dividing $2\#\Sh(E/\Q) \prod_{\ell\neq p} c_\ell(E)$. Furthermore, the set of primes at which $\mu>0$ or $\lambda>0$ is infinite if and only if $\Pi_{an,E}$ is infinite.
\end{theorem}

 \par  Recall that for $x>0$, the set $\Pi_{an, E}(x)$ is the subset of $\Pi_{an,E}$ consisting of primes $p\leq x$. There are results about the asymptotic growth of $\#\Pi_{an, E}(x)$ as a function in $x$. The conjecture of Lang and Trotter \cite{lang2006frobenius} predicts that there is a constant $C\geq 0$ (depending on $E$ and independent of $p$) such that 
\begin{equation}\label{LTconj}\#\Pi_{an, E}(x)\sim C \frac{\sqrt{x}}{\op{log} x}.\end{equation}
This is only conjectural, in fact, the best known upper bound, due to V.~K.~Murty (see \cite{murty1997modular}) asserts that
\[\#\Pi_{an, E}(x)\ll \frac{x(\log \log x)^2}{(\log x)^2}.\] In \eqref{LTconj}, the constant $C$ may actually be $0$, and there are examples of elliptic curves $E$ such that the entire set of anomalous primes $\Pi_{an,E}$ is finite. Such elliptic curves are known as \emph{finitely anomalous curves}, and were studied by Ridgdill in her thesis \cite{ridgdill2010frequency}. 
\par In studying the fine Selmer group, anomalous primes no longer play a role since $\widetilde{E}(\F_p)$ is not a term in the Euler characteristic formula for the leading term of the characteristic element. Instead we turn our attention to local torsion primes. Let $E$ be an elliptic curve over $\Q$, and $p$ a prime number. Recall that \[D=D(E,p):=\op{cok} \left(E(\Q)\otimes \Z_p\longrightarrow \widehat{E(\Q_p)}\right),\] where $\widehat{E(\Q_p)}$ is the $p$-adic completion of $E(\Q_p)$. The torsion subgroup $\op{Tors}_{\Z_p}D$ features in the numerator of \eqref{leading term}. The following conjecture on the distribution of local torsion primes is supported by heuristics.
\begin{conjecture}[David-Weston]\label{conj dw}
Let $E$ be an elliptic curve over $\Q$ without complex-multiplication. Then, the set of local torsion primes is finite. 
\end{conjecture}
The following heuristic is taken from \cite[p.1]{davidweston2008}. Suppose $p$ is a local torsion prime, i.e., $E(\Q_p)[p]\neq 0$ such that $E$ has good reduction at $p$. This in particular implies that $E(\F_p)[p]\neq 0$. If $p\geq 7$, then it follows from the Weil bound that $a_p(E)=1$. On the other hand, $a_p(E)=1$ occurs with probability $\frac{1}{4\sqrt{p}}$. Given an elliptic curve $\widetilde{E}$ over $\F_p$ satisfying some additional conditions, $1$ out of every $p$ lifts $E$ of $\widetilde{E}$ to $\Z_p$ has a local torsion point, see \cite[Corollary 3.5]{davidweston2008}. The sum $\sum_n \frac{1}{4\sqrt{p}\cdot p}=\frac{1}{4}\sum_p p^{-3/2}$ is finite.

\par On the other hand, Conjecture 1.2 in \cite{wuthrich2007iwasawa} states that the fine Selmer group $\Rfine{\Q_{\op{cyc}}}$ associated to an elliptic curve with $\op{rank} E(\Q)\leq 1$ is finite for all but finitely many primes $p$. More generally, the $\Z_p$-corank of $\Rfine{\Q_{\op{cyc}}}$ is expected to be equal to the $\Z_p$-corank of $\Rfine{\Q}$ for all but finitely many primes $p$. We state the conjecture slightly differently from its original formulation.

\begin{conjecture}[Wuthrich] \label{conj wuthrich}
Let $E$ be an elliptic curve over $\Q$, then for all but finitely many primes $p$ of good reduction, the fine Selmer group $\Rfine{\Q_{\op{cyc}}}$ satisfies $\mu=0$ and $\lambda=\op{max}\left(\op{rank} E(\Q)-1, 0\right)$.  
\end{conjecture}

We specialize to the case when $\op{rank} E(\Q)=0$ to obtain the following easy observation.

\begin{proposition}
Suppose that $E$ is an elliptic curve over $\Q$ with $\op{rank} E(\Q)=0$. Assume that the following conditions are satisfied
\begin{enumerate}
\item $E$ does not have complex multiplication,
\item Conjecture \ref{conj dw} holds for $E$,
\item $\Sh(E/\Q)$ is finite.
\end{enumerate}
Then, Conjecture \ref{conj wuthrich} holds for $E$.
\end{proposition}
\begin{proof}
Note that $\op{Reg}\left(\Rfine{\Q}\right)=1$ since $\op{rank} E(\Q)=0$. The conditions of Theorem \ref{wuthrich thm} are satisfied. Since $\op{rank} E(\Q)=0$, the order of vanishing of $f(T)=\op{char} \left(\Rfine{\Q_{\op{cyc}}}\right)$ at $T=0$ is $0$. We appeal to Wuthrich's formula \eqref{leading term} for the leading term $a_0$. Note that $a_0$ is a unit in $\Z_p$ if and only if $\mufine=0$ and $\lafine=0$. Conjecture \ref{conj wuthrich} in this context states that $a_0=a_0(E,p)$ is a $p$-adic unit for all but finitely many primes $p$. By Corollary \ref{ cor 2.9}, $a_0$ is a $p$-adic unit if an only if 
\begin{enumerate}
 \item $p\nmid c_\ell(E)$ for all primes $\ell\neq p$,
    \item $\Zhe_{p^\infty} (E/\Q)=0$,
    \item $\op{Tors}_{\Z_p}(D)=0$.
\end{enumerate}
Theorem \ref{theorem 2.5} implies that $\Zhe_{p^\infty} (E/\Q)=0$ if and only if $\Sh(E/\Q)[p^\infty]=0$. There are only finitely many primes $p$ which divide the Tamagawa product. Since it is assumed that $\Sh(E/\Q)$ is finite, it follows that $\Zhe_{p^\infty} (E/\Q)=0$ for all but finitely many primes. If the set of local torsion primes is finite, then, $\op{Tors}_{\Z_p}(D)=0$ for all but finitely many primes $p$. As a result, the above conditions imply that Conjecture \ref{conj wuthrich} holds for $E$.
\end{proof}
Let $E$ be an elliptic curve over $\Q$ with good reduction at $p$, in other words, $E_{/\Q_p}$ admits a smooth local model $\mathcal{E}_{/\Z_p}$. Given a finitely generated abelian group $M$, we set $\op{rank}_pM:=\op{dim}\left(M\otimes \F_p\right)$ to be the $p$-rank of $M$.

\begin{lemma}\label{init lemma}
Let $\mathcal{E}$ be an elliptic curve over $\Z_p$, then, 
\[\op{rank}_p \left(\mathcal{E}(\Z/p^2\Z)\right)=\begin{cases} 1 \text{ if }E(\Q_p)[p]=0;\\
2\text{ if }E(\Q_p)[p]\neq 0.
\end{cases}\]
\end{lemma}
\begin{proof}
The assertion is a special case of \cite[Lemma 3.1]{davidweston2008}. 
\end{proof}

Given a pair $(a,b)$, we let $E_{a,b}$ be the elliptic curve defined by $y^2=x^3+ax+b$.

\begin{proposition}\label{fibers order p}
Let $E_{a,b}$ be an elliptic curve associated to a pair $(a,b)\in \F_p\times \F_p$ and such that the j-invariant $j(E_{a,b})\neq 0,1728$. Then, there are $p$ distinct pairs $(A_i, B_i)\in \Z/p^2\times \Z/p^2$ such that $(A_i, B_i)\equiv (a,b)\mod{p}$ and \[\op{rank}_p E_{A_i, B_i}(\Z/p^2)=2.\]
\end{proposition}
\begin{proof}
See \cite[Proposition 3.4]{davidweston2008} for a proof of the result.
\end{proof}

\begin{definition}\label{ definition Cp Sp}
Let $\mathfrak{S}_p$ be the subset of $\F_p\times \F_p$ consisting of all pairs $(a,b)$ such that 
\[\Delta_{a,b}\neq 0\text{ and }E_{a,b}(\F_p)[p]\neq 0.\] Let $\mathfrak{S}_p^j$ be the subset of $\mathfrak{S}_p$ consisting of pairs $(a,b)$ such that $j(E_{a,b})=j$. Set $\mathfrak{A}_p$ be the set of pairs $(A,B)\in \Z/p^2\times \Z/p^2$ such that
\[p\nmid \Delta_{A,B} \text{ and }\op{rank}_p\left(E_{A,B}(\Z/p^2)\right)=2.\]
Given $j\in \F_p$, let $\mathfrak{A}_p^j\subseteq \mathfrak{A}_p$ consist of the pairs for which $j(E_{A,B})\equiv j\mod{p}$.
\end{definition}

We would like to estimate $\#\mathfrak{S}_p$ and $\#\mathfrak{A}_p$. Note that by definition $\mathfrak{A}_p$ reduces mod-$p$ to $\mathfrak{S}_p$, and by Proposition \ref{fibers order p}, the fibers of the reduction map $\mathfrak{A}_p^j\rightarrow \mathfrak{S}_p^j$ have order $p$ provided $j\neq 0, 1728$. The quantity $\#\mathfrak{S}_p$ can be expressed in terms of Hurwitz class numbers associated to integral binary quadratic forms with fixed discriminant.
\par We follow the notation and conventions of \cite[section 2]{schoof1987nonsingular}. Let $\Delta\in \Z_{<0}$ with $\Delta\equiv 0,1\mod{4}$, and set
\[B(\Delta):=\left\{ax^2+bxy+cy^2\in \Z[x,y]\mid  a>0, b^2-4ac=\Delta\right\}.\] The group $\op{SL}_2(\Z)$ acts on $\Z[x,y]$, where a matrix $\sigma=\mtx{p}{q}{r}{s}$ sends $x\mapsto (px+qy)$ and $y\mapsto (rx+sy)$. Thus, if $f=ax^2+bxy+cy^2$, the matrix $\sigma$ acts by
\[f\circ \sigma= a(px+qy)^2+b(px+qy)(rx+sy)+c(rx+sy)^2.\] It can be checked that $B(\Delta)$ is stable under the action of $\op{SL}_2(\Z)$ and $B(\Delta)/\op{SL}_2(\Z)$ is finite, see \emph{loc. cit.} for additional details.
\begin{definition}
The \emph{Hurwitz class number} $H(\Delta)$ is the order of $B(\Delta)/\op{SL}_2(\Z)$.
\end{definition}
Given an elliptic curve $E_{/\F_p}$, the Frobenius element $\varphi\in \op{End}(E)$ satisfies the equation $\varphi^2-t\varphi+p=0$. Here, $t$ is an integer given by $\# E(\F_p)=p+1-t$. According to the well-known Hurwitz bound, we have that $|t|\leq 2 \sqrt{p}$. Note that $p$ divides $t$ if and only if $E$ is supersingular. 
\begin{definition} (cf. \cite[Definition 4.1]{schoof1987nonsingular})
Two elliptic curves $E$ and $E'$ over $\F_p$ are \emph{isogenous over $\F_p$} if \[\#E(\F_p)=\#E'(\F_p).\] Let $I(t)$ be the isogeny class of elliptic curves over $\F_p$ with $\#E(\F_p)=p+1-t$. Let $N(t)$ denote the number of $\F_p$-isomorphism classes of elliptic curves in $I(t)$.
\end{definition}
Consider the special case when $\#E(\F_p)$ is divisible by $p$, in other words, $t\equiv 1\mod{p}$. Note that since $|t|\leq  2\sqrt{p}$, it follows that if $p\geq 7$, then, $E(\F_p)$ contains a non-zero $p$-torsion point precisely when $t=1$. However, if $p\leq 5$, then the two possibilities are $t=-p+1$, in which case, $\# E(\F_p)=2p$, and $t=1$, i.e., $\#E(\F_p)=p$.
\par Let $\bar{\mathfrak{S}}_p$ be the set of isomorphism classes of elliptic curves $E$ over $\F_p$ such that $E(\F_p)[p]\neq 0$. Thus, we have that
\[\#\bar{\mathfrak{S}}_p=\begin{cases}N(1)&\text{ for }p\geq 7,\\
N(-p+1)+N(1) &\text{ for }p\leq 5.\end{cases}\] Waterhouse showed that the number of elliptic curves in a given isogeny class can be computed in terms of certain class numbers, see \cite{waterhouse1969abelian}. These results were later reformulated in terms of Hurwitz class numbers by Schoof.
\begin{theorem}[Waterhouse, Schoof]\label{waterhouse schoof theorem}
Let $t$ be an integer such that $p\nmid t$ and $t^2<4p$, then, $N(t)=H(t^2-4p)$.
\end{theorem}

\begin{proof}
The reader is referred to \cite[Theorem 4.6]{schoof1987nonsingular} for a proof of the result.
\end{proof}

\begin{corollary}\label{ estimate on bar Sp}
Let $p$ be a prime and let $\bar{\mathfrak{S}}_p$ be as above. Then, we have that 
\[\#\bar{\mathfrak{S}}_p=\begin{cases}H(1-4p)&\text{ for }p\geq 7,\\
H(p^2+1-6p)+H(1-4p) &\text{ for }p\leq 5.\end{cases}\]
\end{corollary}
\begin{proof}
From the discussion above Theorem \ref{waterhouse schoof theorem}, 
\[\#\bar{\mathfrak{S}}_p=\begin{cases}N(1)&\text{ for }p\geq 7,\\
N(-p+1)+N(1) &\text{ for }p\leq 5.\end{cases}\] The result then follows from Theorem \ref{waterhouse schoof theorem}.
\end{proof}
\begin{proposition}\label{Cp estimate prop}
Let $p$ be any prime and $\mathfrak{S}_p$ be as above. Then, we have that 
\[\# \mathfrak{S}_p\leq \left(\frac{p-1}{2}\right)\left\{ z_p H(p^2+1-6p)+H(1-4p) \right\}< Cp^{3/2} \op{log}p (\op{log log } p)^2,\] where $C>0$ is an effective constant (independent of $p$) and $z_p=\begin{cases} 1 &\text{ if }p\leq 5,\\
0 &\text{ if }p\geq 7.
\end{cases}$
\end{proposition}
\begin{proof}
The elliptic curves $E_{a,b}$ and $E_{a',b'}$ are isomorphic over $\F_p$ if there exists $c\in \F_p^\times$ such that 
\[a'=c^4a \text{ and } b'=c^6 b.\] Therefore, the number of elliptic curves $E_{a',b'}$ isomorphic to a given elliptic curve $E_{a,b}$ is at most $\left(\frac{p-1}{2}\right)$. Therefore, we have that
\[\begin{split}
    \# \mathfrak{S}_p\leq \left(\frac{p-1}{2}\right) \# \bar{\mathfrak{S}}_p<Cp^{3/2} \op{log} p (\op{log log} p)^2,
\end{split}\] where the first inequality is the assertion of Corollary \ref{ estimate on bar Sp}, and the second follows from a well known bound on Hurwitz numbers \cite[Proposition 1.8]{Lenstra_annals}.
\end{proof}
\begin{corollary}\label{cor frakCp}
For $\mathfrak{A}_p$ be as in Definition \ref{ definition Cp Sp}, we have that 
\begin{equation}\label{bound on Up}\# \frac{\mathfrak{A}_p}{p^4}<\frac{2}{p} + C p^{-3/2}\op{log}p (\op{log log} p)^2,\end{equation} where $C>0$ is an effective constant.
\end{corollary}
\begin{proof}
We subdivide $\mathfrak{A}_p$ into two disjoint sets $\mathfrak{A}_p^{(1)}$ and $\mathfrak{A}_p^{(2)}$. Set $\mathfrak{A}_p^{(1)}:=\mathfrak{A}_p^0\cup \mathfrak{A}_p^{1728}$, i.e., the pairs $(A,B)\in \mathfrak{A}_p$ such that the $j$-invariant is either $0$ or $1728$ modulo $p$. Note that $\mathfrak{A}_p^{(1)}$ consists of all pairs $(A,B)\in \mathfrak{A}_p$ such that $p\mid AB$. Clearly, the number of pairs $(A,B)$ such that $p\mid A$ is $p^3$. Thus we have the following obvious upper bound: \begin{equation}\label{cor 3.13 e1}\#\mathfrak{A}_p^{(1)}<2 p^3.\end{equation} Let $\mathfrak{A}_p^{(2)}$ be the complement of $\mathfrak{A}_p^{(1)}$ in $\mathfrak{A}_p$. For any pair $(A,B)\in \mathfrak{A}_p^{(2)}$, we have by construction that $j(E_{A,B})\not \equiv 0, 1728\mod{p}$. Proposition \ref{fibers order p} asserts that given a pair $(a, b)\in \mathfrak{S}_p$ with $j$-invariant not equal to $0$ or $1728$, there are $p$ lifts $(A,B)\in \mathfrak{A}_p$ of $(a,b)$. Therefore, we have the following upper bound: $\# \mathfrak{A}_p^{(2)}\leq p \# \mathfrak{S}_p$. Invoking Proposition \ref{Cp estimate prop}, we obtain the following bound:
\begin{equation}\label{cor 3.13 e2}\# \mathfrak{A}_p^{(2)}< Cp^{5/2} \op{log}p (\op{log log } p)^2.\end{equation}
Putting together \eqref{cor 3.13 e1} and \eqref{cor 3.13 e2}, we obtain
\[ \frac{\#\mathfrak{A}_p}{p^4}=\frac{\#\mathfrak{A}_p^{(1)}}{p^4}+\frac{\#\mathfrak{A}_p^{(2)}}{p^4}<\frac{2}{p} + C p^{-3/2}\op{log}p (\op{log log} p)^2.\]
\end{proof}
\begin{remark}
In the proof of the above result, we have invoked the obvious upper bound for $\#\mathfrak{A}_p^{(1)}$, which contributes to the term $\frac{2}{p}$ in the estimate \eqref{bound on Up}. Given a pair $(A,B)$ such that $p\mid AB$, note that the pair $(A,B)$ is contained in $\mathfrak{A}_p^{(1)}$ precisely when $(\bar{A}, \bar{B})=(A,B)\mod{p}$ is contained in $\mathfrak{S}_p$. Let $\mathfrak{S}_p^{\star}$ be the subset of $\mathfrak{S}_p$ consisting of all pairs $(a,b)$ such that either $a$ or $b$ is $0$. In the above proof, we have used the obvious bound $\#\mathfrak{S}_p^{\star}< 2p$ to obtain that 
\[\# \mathfrak{A}_p^{(1)}\leq p^2\#\mathfrak{S}_p^{\star}< 2p^3.\] The main term $2/p$ in the estimate $\eqref{bound on Up}$ can be replaced by $\#\mathfrak{S}_p^{\star}/p^2$. Thus, a better understanding of $\#\mathfrak{S}_p^{\star}$ would potentially give a better estimate.
\end{remark}
\section{Results for a fixed prime and varying elliptic curve}\label{s 4}
\par Let $E_{/\Q}$ be an elliptic curve. Note that $E$ is defined by a Weierstrass equation $E:y^2=x^3+ax+b$, where the integers $a$ and $b$ such that for all primes $\ell$, $\ell^6\nmid b$ whenever $\ell^4\mid a$. Such a Weierstrass equation is unique and referred to as the \emph{minimal Weierstrass equation}. The \emph{naive height} of $E$ is defined to be the quantity $H(E):=\op{max}\left\{|a|^3, b^2\right\}$.
\begin{definition}\label{density def}Let $\mathcal{E}$ be the set of isomorphism classes of elliptic curves defined over $\Q$. Given a subset $\mathcal{S}$ of $\mathcal{E}$, and $x>0$, we set
\[\mathcal{S}_{<x}:=\{E\in \mathcal{S}\mid H(E)<x\},\]
and $\frak{d}(\mathcal{S},x):=\frac{\# \mathcal{S}_{<x}}{\# \mathcal{E}_{<x}}$.
We say that $\mathcal{S}$ has upper (resp. lower) density $\delta$ if $\limsup_{x\rightarrow \infty} \frak{d}(\mathcal{S},x) =\delta$ (resp. $\liminf_{x\rightarrow \infty} \frak{d}(\mathcal{S},x) =\delta$). If 
\[\lim_{x\rightarrow \infty} \frak{d}(\mathcal{S},x) =\delta,\] then, we say that $\mathcal{S}$ has density $\delta$. Part of the requirement here is that the above limit exists. We let $\bar{\mathfrak{d}}(\mathcal{S})$ (resp. $\underline{\mathfrak{d}}(\mathcal{S})$) be the upper (resp. lower) density of $\mathcal{S}$. The density is denoted $\mathfrak{d}(\mathcal{S})$. Note that when $\mathfrak{d}(\mathcal{S})$ is defined, so are $\underline{\mathfrak{d}}(\mathcal{S})$ and $\bar{\mathfrak{d}}(\mathcal{S})$, and  
\[\mathfrak{d}(\mathcal{S})=\underline{\mathfrak{d}}(\mathcal{S})=\bar{\mathfrak{d}}(\mathcal{S}).\]
\end{definition}
\par The conjecture on rank distribution by Katz and Sarnak \cite{katz1999random} predicts that when ordered by height, $50\%$ of elliptic curves have rank $0$, the other $50\%$ of rank $1$ and $0\%$ of rank $>1$. In other words, let $\mathcal{E}^0, \mathcal{E}^1$ and $\mathcal{E}^{>1}$ be the subset of elliptic curves of rank $0,1$ and $>1$ respectively. Then, the conjecture predicts that $\mathfrak{d}(\mathcal{E}^0)=\mathfrak{d}(\mathcal{E}^1)=\frac{1}{2}$, and $\mathfrak{d}(\mathcal{E}^{>1})=0$.
\par Certain unconditional results have been proven by Bhargava and Shankar who study the average size of the the $p$-Selmer group $\op{Sel}_p(E/\Q)$. This is the $p$-torsion subgroup of the $p$-primary Selmer group $\op{Sel}_{p^\infty}(E/\Q)$.  For instance, in the preprint \cite{bhargava2013average}, it is shown that the average size of the $5$-Selmer group is $6$ as $E$ varies over $\mathcal{E}$. This result implies that $\underline{\mathfrak{d}}(\mathcal{E}^0)\geq \frac{1}{5}$ and $\bar{\mathfrak{d}}(\mathcal{E}^{>1})\leq \frac{1}{5}$.

\par In this section, we fix a prime number $p\geq 5$ and study the following question.

\begin{question}
As $E$ varies over all elliptic curves defined over $\Q$, what can be said about the proportion of elliptic curves $E$ for which the following equivalent conditions satisfied:
\begin{enumerate}
    \item $\Rfine{\Q_{\op{cyc}}}$ has finite cardinality.
    \item The $\mu$ and $\lambda$-invariants of $\Rfine{\Q_{\op{cyc}}}$ are both equal to $0$.
    \item The leading term $f(0)$ of the characteristic element $f(T)$ of $\Rfine{\Q_{\op{cyc}}}$ is a unit in $\Z_p$.
\end{enumerate}
\end{question}

Note that when the Selmer group $\op{Sel}_{p^\infty}(E/\Q_{\op{cyc}})$ is finite, it is in fact equal to $0$, see \cite[Corollary 3.6]{KR21}. This is only possible when the Mordell-Weil rank of $E$ is $0$ and $E$ has ordinary reduction at $p$ (for further details, the reader may refer to \cite[section 4]{KR21}). Since the fine Selmer group is a subgroup of the Selmer group, it follows that if $\op{Sel}_{p^\infty}(E/\Q_{\op{cyc}})=0$, then so is $\Rfine{\Q_{\op{cyc}}}=0$. However, it is indeed possible for the fine Selmer group $\Rfine{\Q_{\op{cyc}}}$ to be finite, while the entire Selmer group $\op{Sel}_{p^\infty}(E/\Q_{\op{cyc}})$ is infinite. As the following example shows, this may be so even in the special case when the elliptic curve $E$ has Mordell-Weil rank $0$ and good ordinary reduction at $p$.

\begin{example}
We refer to \cite[section 9.1]{wuthrich2007iwasawa}, consider the elliptic curve 
\[E:y^2+y=x^3-x^2-10x-20,\] and consider the prime $p=5$. The $\mu$-invariant of the Selmer group $\op{Sel}_{p^\infty}(E/\Q_{\op{cyc}})$ is $1$, in fact, 
\[\op{Sel}_{p^\infty}(E/\Q_{\op{cyc}})^\vee=\Lambda/p.\] On the other hand, $\Rfine{\Q_{\op{cyc}}}$ is finite and non-zero. In particular, the $\mu$ and $\lambda$-invariants of the fine Selmer group are zero, however, $\Rfine{\Q_{\op{cyc}}}\neq 0$.
\end{example}
More generally, if $\Rfine{\Q_{\op{cyc}}}$ is finite and non-zero at a prime of good ordinary reduction, then the Selmer group $\op{Sel}_{p^\infty}(E/\Q_{\op{cyc}})$ must be infinite, since otherwise it would be zero.

\subsection{Statistics for the Selmer group}
Before we proceed to prove results for the fine Selmer group, let us briefly recall the statistical results for the Selmer group $\op{Sel}_{p^\infty}(E/\Q_{\op{cyc}})$ proved by Kundu and the first named author in \cite[section 4]{KR21}. Some of these results may be marginally refined using \cite[Lemma 6.4]{ray2021arithmetic}. In this subsection, we summarize the results from these sources. Let $\mathfrak{F}_p$ be the the set of elliptic curves $E\in \mathcal{E}$, such that 
\begin{enumerate}
    \item $\op{rank} E(\Q)=0$, 
    \item $E$ has good ordinary reduction at $p$,
    \item $\op{Sel}_{p^\infty}(E/\Q_{\op{cyc}})\neq 0$. 
\end{enumerate}

Let $\mathcal{S}_p$ be the set of elliptic curves $E_{/\Q}$ such that $\Sh(E/\Q)[p^\infty]\neq 0$. A conjecture of Delaunay \cite{delaunay2007heuristics} predicts that $\bar{\mathfrak{d}}(\mathcal{S}_p)=1-\prod_{i\geq 1}\left(1-\frac{1}{p^{2i-1}}\right)=\frac{1}{p}+\frac{1}{p^3}-\frac{1}{p^4}+\dots$. The heuristics supporting this conjecture have been further refined by Poonen and Rains, see \cite{poonen2012random}. \begin{conjecture}[Del-p]\label{Del}
The proportion of all elliptic curves such that $E$ for which $p\mid \# \Sh(E/\Q)$ is given by
\[\left(1-\prod_{i\geq 1} \left(1-\frac{1}{p^{2i-1}}\right)\right).\]
\end{conjecture}
In preparation of the next result, recall the definition of the set $\mathfrak{S}_p$ from Definition \ref{ definition Cp Sp}.
\begin{theorem}(Kundu $\&$ R.) \label{kundu ray theorem}
Let $p\geq 5$ be a prime number and $\mathfrak{F}_p$ be defined as above. Assume $\Sh(E/\Q)[p^\infty]$ is finite for all elliptic curves $E_{/\Q}$ and that the conjecture (Del-p) is satisfied. Then, there is an effective constant $C_1>0$ such that the following bounds are satisfied:
\begin{equation}\label{selmer main formula}
\begin{split}\bar{\mathfrak{d}}(\mathfrak{F}_p)\leq & \zeta(10)\frac{\#\mathfrak{S}_p}{p^2}+\left(1-\prod_{i\geq 1} \left(1-\frac{1}{p^{2i-1}}\right)\right)+(\zeta(p)-1),\\
< & C_1 \frac{\op{log} p (\op{log log}p)^2}{\sqrt{p}}+\left(1-\prod_{i\geq 1} \left(1-\frac{1}{p^{2i-1}}\right)\right)+(\zeta(p)-1).
\end{split}\end{equation}
\end{theorem}
Moreover, the effective constant $C_1$ can be taken to be $\zeta(10) C$, where $C$ is the constant from Proposition \ref{Cp estimate prop}.
\begin{proof}
The result is essentially a rephrasing of \cite[Theorem 4.3]{KR21}, however, thanks to the minor update in the statement taken from \cite{ray2021arithmetic}, and for the convenience of the reader, we summarize the proof here. Let $E$ be an elliptic curve with good ordinary reduction at $p$, and $f(T)$ be the characteristic element of $\op{Sel}_{p^\infty}(E/\Q_{\op{cyc}})^{\vee}$. Express $f(T)$ as a power series. Suppose that $E\in \mathfrak{F}_p$, then, in particular, $E$ has good ordinary reduction at $p$ and $\op{rank} E(\Q)=0$. On the other hand, it is assumed that $\Sh(E/\Q)[p^\infty]$ is finite, and therefore the Selmer group $\op{Sel}_{p^\infty}(E/\Q)$ is finite. It follows that the constant term $f(0)$ is non-zero $p$-adic integer. By the well known Euler characteristic formula,
\[f(0)\sim \frac{\# \Sh(E/\Q)[p^{\infty}]\times \left(\prod_{l}c_l(E)\right)\times \left(\# \widetilde{E}(\F_p)\right)^2}{\#\left(E(\Q)[p^{\infty}]\right)^2}.\]
In particular, if the three terms $\# \Sh(E/\Q)[p^{\infty}]$, $\left(\prod_{l}c_l(E)\right)$ and $\#\widetilde{E}(\F_p)$ are units in $\Z_p$, then $f(0)$ will be a unit in $\Z_p$. We study the average behaviour of the following quantities for fixed $p$ and varying $E\in \mathcal{E}$,
\begin{enumerate}
    \item $s_p(E):=\#\Sh(E/\Q)[p^{\infty}]$, 
    \item $\tau_p(E):=\prod_l c_l^{(p)}(E)$,
    \item $\delta_p(E):=\#\left(\widetilde{E}(\F_p)[p]\right)$.
\end{enumerate}
Let $\mathcal{E}_1, \mathcal{E}_2$ and $\mathcal{E}_3$ to be the set of \emph{all} elliptic curves $E_{/\Q}$ with good reduction at $p$, such that $p$ divides $s_p(E)$, $\tau_p(E)$ and $\delta_p(E)$ respectively. Note that if $f(0)$ is a $p$-adic unit, then $\op{Sel}_{p^\infty}(E/\Q_{\op{cyc}})=0$. Thus if $E\in \mathfrak{F}_p$, then $f(0)$ is not a unit in $\Z_p$, hence, $E\in \mathcal{E}_1\cup \mathcal{E}_2\cup \mathcal{E}_3$. From the containment $\mathfrak{F}_p\subseteq \bigcup_{i=1}^3 \mathcal{E}_i$, we obtain the following bound:
\[\bar{\mathfrak{d}}(\mathfrak{F}_p)\leq \sum_{i=1}^3 \bar{\mathfrak{d}}(\mathcal{E}_i).\]
Since we assume that (Del-p) is satisfied, we have that 
\[\bar{\mathfrak{d}}(\mathcal{E}_1)\leq \left(1-\prod_{i\geq 1} \left(1-\frac{1}{p^{2i-1}}\right)\right).\]
It follows from \cite[Theorem 4.2]{KR21} or \cite[Corollary 8,8]{hatley2021statistics} that 
\[\bar{\mathfrak{d}}(\mathcal{E}_2)\leq \zeta(p)-1.\]
Next, it follows from \cite[Theorem 4.14]{KR21} and Proposition \ref{Cp estimate prop} (or \cite[Lemma 6.4]{ray2021arithmetic}) that
\[\bar{\mathfrak{d}}(\mathcal{E}_3)\leq \zeta(10) \frac{\# \mathfrak{S}_p}{p^2}<C_1 \frac{\op{log} p(\op{loglog} p)^2}{\sqrt{p}},\] where $C_1=\zeta(10) C$. Putting everything together, the result follows.
 \end{proof}
 \begin{remark}
The result shows that when $p$ is suitably large the density of the set $\mathfrak{F}_p$ is quite small.
 \begin{itemize}
     \item Note that $\sqrt{p}$ grows faster than $\op{log} p (\op{loglog} p)^2$ as $p\rightarrow \infty$, hence, the term $C_1 \frac{\op{log} p(\op{log log} p)^2}{\sqrt{p}}$ goes to $0$ as $p\rightarrow \infty$. However, the exact values of $\# \frac{\mathfrak{S}_p}{p^2}$ can be calculated for small enough primes, and for this, we refer to Table \ref{tab:1} (or, \cite[Table 2]{KR21II}).
     \item The function \[f(p):= \left(1-\prod_{i\geq 1} \left(1-\frac{1}{p^{2i-1}}\right)\right)=\frac{1}{p}+\frac{1}{p^3}-\frac{1}{p^4}+\dots\] is a function whose domain is the set of prime numbers $p$. As a function of $p$, it is asymptotic to $\frac{1}{p}$. In other words, 
     \[\lim_{p\rightarrow \infty} \frac{f(p)}{p^{-1}}=1,\] where the limit is over primes $p$. \par This is easy to prove, however, we provide details here. First, it is easy to see that for any integer $m\geq 1$,
     \[1\geq \prod_{i\geq m}\left(1-\frac{1}{p^{2i-1}}\right)\geq 1-\sum_{i\geq m} \frac{1}{p^{2i-1}}.\]
     Therefore, we have the following bounds
     \begin{equation}\label{bound 1}f(p)= 1-\left(1-\frac{1}{p}\right)\prod_{i\geq 2}\left(1-\frac{1}{p^{2i-1}}\right)\geq 1-\left(1-\frac{1}{p}\right)=\frac{1}{p},\end{equation}
     \begin{equation}\label{bound 2}f(p)= 1-\prod_{i\geq 1}\left(1-\frac{1}{p^{2i-1}}\right)\leq \sum_{i\geq 1} \frac{1}{p^{2i-1}}= \frac{1}{p}+\frac{1}{p(p^2-1)}.\end{equation}
     Combining \eqref{bound 1} and \eqref{bound 2}, we find that 
     \[1\leq \frac{f(p)}{p^{-1}}\leq 1+\frac{1}{(p^2-1)},\] and it thus follows that $f(p)\sim \frac{1}{p}$. Therefore, in particular, $f(p)$ is asymptotically smaller than the term $C_1 \frac{\op{log} p(\op{log log} p)^2}{\sqrt{p}}$ as a function of $p$.
     \item It is easy to see that 
 \[\zeta(p)-1=2^{-p}+\sum_{n\geq 3} n^{-p}< 2^{-p}+\int_{2}^\infty x^{-p}dx=2^{-p}\left(\frac{p+1}{p-1}\right).\]
 \end{itemize}
  \end{remark}
 Thus, Theorem \ref{kundu ray theorem} shows that $\bar{\mathfrak{d}}(\mathfrak{F}_p)\rightarrow 0$ as $p\rightarrow \infty$, and in fact, as a function of $p$ is $O\left(\frac{\op{log}p (\op{log log} p)^2}{\sqrt{p}}\right)$, thereby indicating that given a prime $p$, most elliptic curves $E_{/\Q}$ of rank $0$ and good ordinary reduction at $p$ have $\op{Sel}_{p^\infty}(E/\Q_{\op{cyc}})=0$. In fact, the above remark shows that the contribution of the Tamagawa factors is very little in comparison to that of anomalous primes.

 \subsection{Statistics for the Fine Selmer group in the rank-zero case}
 \par We now turn our attention to the fine Selmer group. In this subsection, we study the case when $\op{rank} E(\Q)=0$. Let $p\geq 5$ be a prime number and let $\mathfrak{B}_p$ be the set of elliptic curves $E_{/\Q}$ such that the following conditions are satisfied:
 \begin{enumerate}
     \item $E$ has good reduction at $p$, 
     \item $\op{rank} E(\Q)=0$,
     \item $\Rfine{\Q_{\op{cyc}}}$ is infinite (equivalently, either $\mufine>0$ or $\lafine>0$).
 \end{enumerate}
 Thus, we no longer impose the good ordinary condition at $p$.
 \begin{theorem}\label{th 4.6}
 Let $p\geq 5$ be a prime and $\mathfrak{B}_p$ be defined as above. Assume that $\Sh(E/\Q)[p^\infty]$ is finite for all elliptic curves $E_{/\Q}$ and that the conjecture (Del-p) is satisfied. Then, we have the following bound on the upper density of $\mathfrak{B}_p$,
 \begin{equation}\label{main bound fine Selmer}
 \begin{split}
     \bar{\mathfrak{d}}(\mathfrak{B}_p)\leq & \zeta(10)\frac{\#\mathfrak{A}_p}{p^4}+ \left(1-\prod_{i\geq 1} \left(1-\frac{1}{p^{2i-1}}\right)\right)+(\zeta(p)-1),\\
     \leq & \frac{2 \zeta(10)}{p} + C_1 \frac{ \op{log}p (\op{log log} p)^2}{p^{3/2}}+\left(1-\prod_{i\geq 1} \left(1-\frac{1}{p^{2i-1}}\right)\right)+(\zeta(p)-1).\\
     \end{split}
 \end{equation}
 Here $C_1$ is explicit constant given by $C_1=\zeta(10) C$ from Corollary \ref{cor frakCp}.  \end{theorem}
 
 \begin{proof}
Consider the following quantities:\begin{enumerate}
    \item $s_p(E):=\#\Sh(E/\Q)[p^{\infty}]$, 
    \item $\tau_p(E):=\prod_l c_l^{(p)}(E)$,
    \item $\beta_p(E):=\#\left(E(\Q_p)[p]\right)$.
\end{enumerate}
Recall from the proof of Theorem \ref{kundu ray theorem} that $\mathcal{E}_1$ and $\mathcal{E}_2$ consist of all elliptic curves $E_{/\Q}$ such that $p$ divides $s_p(E)$ and $\tau_p(E)$ respectively. Let $\mathcal{E}_4$ be the set of all elliptic curves $E_{/\Q}$, with good reduction at $p$, such that $p$ divides $\beta_p(E)$. For these elliptic curves, $p$ is a local torsion prime over $\Q_p$. It follows from Corollary \ref{criterion for mu=lambda=0} that $\mathfrak{B}_p$ is contained in the union $ \mathcal{E}_1\cup \mathcal{E}_2\cup \mathcal{E}_4 $. Hence, 
\begin{equation}\label{E1 to E4 bound}\bar{\mathfrak{d}}(\mathfrak{B}_p)\leq \bar{\mathfrak{d}}(\mathcal{E}_1)+\bar{\mathfrak{d}}(\mathcal{E}_2)+\bar{\mathfrak{d}}(\mathcal{E}_4).\end{equation} From the proof of Theorem \ref{kundu ray theorem}, we have that 
\[\bar{\mathfrak{d}}(\mathcal{E}_1)+\bar{\mathfrak{d}}(\mathcal{E}_2)\leq \left(1-\prod_{i\geq 1} \left(1-\frac{1}{p^{2i-1}}\right)\right)+(\zeta(p)-1).\]
We prove an upper bound for $\bar{\mathfrak{d}}(\mathcal{E}_4)$. We recall that $\mathfrak{A}_p$ is the set of pairs $(A,B)\in \Z/p^2\times \Z/p^2$ such that $p\nmid \Delta_{A,B} \text{ and }\op{rank}_p\left(E_{A,B}(\Z/p^2)\right)=2$. It follows from Lemma \ref{init lemma} that an elliptic curve $E_{A,B}$ is contained in $\mathcal{E}_4$ if and only the reduction of $(A,B)$ modulo $p^2$ lies in $\mathfrak{A}_p$. Thus it would seem that $\mathfrak{d}(\mathcal{E}_4)$ is simply equal to $\frac{\# \mathfrak{A}_p}{\#\left(\Z/p^2\times \Z/p^2\right)}=\frac{\# \mathfrak{A}_p}{p^4}$, but this is not quite the case, since not all pairs $(A,B)\in \Z\times \Z$ are minimal. It is well known that the proportion of all pairs $(A,B)$ that are minimal is $\frac{1}{\zeta(10)}$, see for instance \cite{CS20}. In greater detail, let $\mathcal{W}$ be the set of all pairs of integers $(A,B)$, and $\mathcal{W}_{<x}$ be the subset of pairs such that the height is $<x$. Note that since each elliptic curve $E_{/\Q}$ admits a unique minimal Weierstrass equation, we may identify $\mathcal{E}$ with a subset of $\mathcal{W}$. It is easy to see that 
\[\lim_{x\rightarrow \infty} \frac{\# \mathcal{E}_{4,<x}}{\# \mathcal{W}_{<x}}=\frac{\# \mathfrak{A}_p}{p^4}\] and since 
\[\lim_{x\rightarrow \infty} \frac{\# \mathcal{E}_{<x}}{\# \mathcal{W}_{<x}}=\frac{1}{\zeta(10)},\] it follows that
\[\mathfrak{d}(\mathcal{E}_4)=\lim_{x\rightarrow x}\frac{\# \mathcal{E}_{4,<x}}{\# \mathcal{E}_{<x}}=\zeta(10)\frac{\# \mathfrak{A}_p}{p^4}.\]

Appealing to Corollary \ref{cor frakCp}, we have that 
\[\mathfrak{d}(\mathcal{E}_4)<\frac{2 \zeta(10)}{p} + C_1 p^{-3/2}\op{log}p (\op{log log} p)^2.\] The result now follows from \eqref{E1 to E4 bound}.
 \end{proof}
 Asymptotically in $p$, the bound \eqref{main bound fine Selmer} is \[\frac{2\zeta(10)+1}{p}+O\left(p^{-3/2}\op{log}p (\op{log log} p)^2\right),\] which is significantly smaller than the asymptotic estimate $O\left(p^{-1/2}\op{log}p (\op{log log} p)^2\right)$ from Theorem \ref{kundu ray theorem}. This tells us that the fine Selmer group is finite a lot more often than the classical Selmer group over $\Q_{\op{cyc}}$, since we have saved on an entire factor of $p^{1/2} \op{log}p (\op{log log} p)^2$ in our estimates.
  \subsection{Statistics for the Fine Selmer group in the rank-one case}
  \par Let $E_{/\Q}$ be an elliptic curve with positive Mordell Weil rank $r=\op{rank}E(\Q)$. Set $\widehat{E(\Q_p)}$ to be the $p$-adic completion of $E(\Q_p)$ given by $\widehat{E(\Q_p)}:=\varprojlim_n \frac{E(\Q_p)}{p^n E(\Q_p)}$. Recall that $D$ is the cokernel of the map $E(\Q)\otimes \Z_p\rightarrow \widehat{E(\Q_p)}$. The group $\widehat{E(\Q_p)}$ decomposes into a direct sum $\Z_p\oplus T$, where $T$ is the torsion subgroup of $\widehat{E(\Q_p)}$ (see \cite[Lemma 2.1, (iv)]{hiranouchi2019local}). Note that $T$ is non-zero when $p$ is a local torsion prime. The summand $\Z_p$ arises from the formal group of $E$. Write $E(\Q)\otimes \Z_p= \Z_p^r\oplus T'$, where $T'$ is a finite group. The map $E(\Q_p)\otimes \Z_p\rightarrow \widehat{E(\Q_p)}$ gives rise to a map 
  \[\Z_p^r\oplus T'\rightarrow \Z_p\oplus T.\] The map on the first factors is denoted by $\phi_E:\Z_p^r\rightarrow \Z_p$. This is the composite of the maps \[\Z_p^r\hookrightarrow \Z_p^r\oplus T'\rightarrow \Z_p\oplus T\rightarrow \Z_p,\] where the first map is the inclusion into the first factor and the last map is the projection on the first factor. 
  
  \begin{lemma}\label{tors d criterion}
  Let $E$ be an elliptic curve and $p$ a prime as above and suppose that $\op{Tors}_{\Z_p}(D)\neq 0$. Then, at least one of the following hold:
  \begin{enumerate}
      \item $p$ is a local torsion prime of $E$, 
      \item $\phi_E$ is not surjective. 
  \end{enumerate}
  \end{lemma}
  \begin{proof}
  Suppose that $p$ is not a local torsion prime. Then, it is clear $D$ is a quotient of $\op{cok} \phi_E$. Hence, if $D\neq 0$, it follows that $\op{cok} \phi_E$ is non-zero. This implies that $\phi_E$ is not surjective.
  \end{proof}
  \par We shift our attention to the case when our elliptic curves $E_{/\Q}$ have Mordell-Weil rank $1$. In this setting, the $p$-adic regulator $\op{Reg}\left(\Rfine{\Q}\right)=1$, and this makes it possible to study this case. Since $r=1$, the map $\phi_E:\Z_p\rightarrow \Z_p$ is surjective if and only if it is an isomorphism. We let $\mathcal{E}_5$ be the set of elliptic curves $E_{/\Q}$ such that 
  \begin{enumerate}
      \item $\op{rank} E(\Q)=1$,
      \item the map $\phi_E$ is \emph{not} an isomorphism.
  \end{enumerate}

  The upper density $\bar{\mathfrak{d}}(\mathcal{E}_5)$ shall play a role in our results. We are unable to prove any satisfactory bound on $\bar{\mathfrak{d}}(\mathcal{E}_5)$ which is why we resort to heuristics. The probablity that an elliptic curve has rank $1$ is expected to be $\frac{1}{2}$. This follows from the rank distribution conjecture. The map $\phi_E: \Z_p\rightarrow \Z_p$ is multiplication by an element $\alpha\in \Z_p$, and is an isomorphism precisely when $\alpha$ is a unit. Thus, the probablity that $\phi_E$ is \emph{not} an isomorphism can be expected to be $\frac{1}{p}$. Thus, the heuristic indicates that $\bar{\mathfrak{d}}(\mathcal{E}_5)=\frac{1}{2p}$. We state our results in terms of the the density $\bar{\mathfrak{d}}(\mathcal{E}_5)$, and do not explicitly assume that this heuristic is correct. Let $p\geq 5$ be a prime number and let $\mathfrak{D}_p$ be the set of elliptic curves $E_{/\Q}$ such that the following conditions are satisfied:
 \begin{enumerate}
     \item $E$ has good reduction at $p$, 
     \item $\op{rank} E(\Q)=1$,
     \item $\Rfine{\Q_{\op{cyc}}}$ is infinite.
 \end{enumerate}
 
 \begin{theorem}\label{th 4.8}
 Let $p\geq 5$ be a prime and $\mathfrak{D}_p$ be defined as above. Assume that $\Sh(E/\Q)[p^\infty]$ is finite for all elliptic curves $E_{/\Q}$ and that the conjecture (Del-p) is satisfied. Then, we have the following bound on the upper density of $\mathfrak{D}_p$,
 \begin{equation}\label{main bound fine Selmer rank 1}
 \begin{split}
     \bar{\mathfrak{d}}(\mathfrak{D}_p)\leq & \bar{\mathfrak{d}}(\mathcal{E}_5)+\zeta(10)\frac{\#\mathfrak{A}_p}{p^4}+ \left(1-\prod_{i\geq 1} \left(1-\frac{1}{p^{2i-1}}\right)\right)+(\zeta(p)-1),\\
     \leq & \bar{\mathfrak{d}}(\mathcal{E}_5)+\frac{2 \zeta(10)}{p} + C_1 \frac{ \op{log}p (\op{log log} p)^2}{p^{3/2}}+\left(1-\prod_{i\geq 1} \left(1-\frac{1}{p^{2i-1}}\right)\right)+(\zeta(p)-1).\\
     \end{split}
 \end{equation}
 Here $C_1$ is explicit constant given by $C_1=\zeta(10) C$ from Corollary \ref{cor frakCp}.  \end{theorem}
 
 \begin{proof}
 Let $E$ be an elliptic curve in $\mathfrak{D}_p$. Consider the following quantities:\begin{enumerate}
    \item $s_p(E):=\#\Sh(E/\Q)[p^{\infty}]$, 
    \item $\tau_p(E):=\prod_l c_l^{(p)}(E)$,
    \item $\beta_p(E):=\#\left(E(\Q_p)[p]\right)$.
\end{enumerate}
Recall that $\mathcal{E}_1$, $\mathcal{E}_2$ and $\mathcal{E}_4$ consist of all elliptic curves $E_{/\Q}$ such that $p$ divides $s_p(E)$, $\tau_p(E)$ and $\beta_p(E)$ respectively. It follows from Corollary \ref{criterion for mu=lambda=0} that $E$ is either contained in the union $\mathcal{E}_1\cup \mathcal{E}_2$ or $\op{Tors}_{\Z_p}(D)\neq 0$. It follows from Lemma \ref{tors d criterion} that if $\op{Tors}_{\Z_p}(D)\neq 0$, then $E\in \mathcal{E}_4\cup \mathcal{E}_5$. Therefore, combining the above observations, we find that $\mathfrak{D}_p$ is contained in the union $\mathcal{E}_1\cup \mathcal{E}_2\cup \mathcal{E}_4\cup \mathcal{E}_5$, and hence, 
\[\bar{\mathfrak{d}}(\mathfrak{D}_p)\leq \bar{\mathfrak{d}}(\mathcal{E}_1)+\bar{\mathfrak{d}}(\mathcal{E}_2)+\bar{\mathfrak{d}}(\mathcal{E}_4)+\bar{\mathfrak{d}}(\mathcal{E}_5).\] From the proof of Theorem \ref{th 4.6}, 
\[\begin{split}\bar{\mathfrak{d}}(\mathcal{E}_1)+\bar{\mathfrak{d}}(\mathcal{E}_2)+\bar{\mathfrak{d}}(\mathcal{E}_4)\leq & \zeta(10)\frac{\#\mathfrak{A}_p}{p^4}+ \left(1-\prod_{i\geq 1} \left(1-\frac{1}{p^{2i-1}}\right)\right)+(\zeta(p)-1)\\
< & C_1 \frac{ \op{log}p (\op{log log} p)^2}{p^{3/2}}+ \left(1-\prod_{i\geq 1} \left(1-\frac{1}{p^{2i-1}}\right)\right)+(\zeta(p)-1),\end{split}\]and the result follows from this.
 \end{proof}
 
\section{Sieve Methods and the distribution of Local torsion primes}\label{s 5}
\par The distribution of local torsion primes is closely connected to the behavior of Iwasawa invariants of elliptic curves. Let $E_{/\Q}$ be an elliptic curve with Mordell-Weil rank zero for which the Tate-Shafarevich group $\Sh(E/\Q)$ is finite. Then Corollary \ref{ cor 2.9} asserts that $\mufine=0$ and $\lafine=0$ for all primes $p\notin \Pi_{lt, E}\cup \mathfrak{S}$. Here, $\mathfrak{S}$ is the set of primes that divide $2\#\Sh(E/\Q)\times \prod_{\ell\neq p} c_\ell(E)$. In particular, if the set $\Pi_{lt,E}$ is finite, then $\mufine=0$ and $\lafine=0$ for all but finitely many primes $p$.
\par In this section we shall fix a number $Y>0$, and prove results about the expectation for $\#\Pi_{lt,E}(Y)$, as $E$ ranges over all elliptic curves over $\Q$. We shall prove our results for elliptic curves on average. For this purpose, the shall employ the use of a certain refined version of the Large sieve inequality due to Bombieri and Davenport \cite{bombieri1969large}. The version that we use has previously been used in various other contexts, see \cite{duke1997elliptic, jones2010almost}. The version we require is slightly more general than the sieve used in the aforementioned works. Given real numbers $a\leq b$, let $[a,b]$ denote the closed interval from $a$ to $b$. Given a tuple of real numbers $(M_1, \dots, M_k)\in \mathbb{R}^k$ and a tuple of positive real numbers $(N_1, \dots, N_k)\in \mathbb{R}_{\geq 0}^k$, the product $\prod_{j=1}^r [M_j, M_j+N_j]$ consists of all tuples of real numbers $(x_1, \dots, x_k)$ such that $M_j\leq x_j\leq M_j+N_j$ for all $j$. We shall refer to such a product of intervals as a \emph{box} in $\mathbb{R}^k$. For each point $n\in B\cap \Z^k$ let $c(n)$ be a complex number. Define a function on $\mathbb{R}^k$ as follows \[S(x):=\sum_{n\in B} c(n) e(n\cdot x),\] where $x\in \mathbb{R}^k$ and $e(z):=e^{2\pi i z}$.

\begin{definition}
Let $A\subset \mathbb{R}^k$ be a finite subset and let $\delta=(\delta_1, \dots, \delta_k)$ be a tuple of positive numbers. We say that $A$ is \emph{$\delta$ well-spaced} if for all $\alpha=(\alpha_1,\dots, \alpha_k)\in A$ and $\beta=(\beta_1,\dots, \beta_k)\in A$ such that $\alpha\neq \beta$, 
\[\op{max}\left\{\frac{|\alpha_i-\beta_i|}{\delta_i}\mid i=1,\dots, k\right\}\geq 1.\]
\end{definition}
In other words, for each pair $(\alpha, \beta)$ such that $\alpha\neq \beta$, there is a coordinate $i$ such that $|\alpha_i-\beta_i|\geq \delta_i$. The following generalization of the Large sieve inequality is due to Huxley, see \cite[Theorem 1]{huxley1968large}. \begin{theorem}[Huxley] \label{huxleythm}
Let $A$ be a $\delta$ well-spaced finite set of vectors in $\mathbb{R}^k$ and let $B=\prod_{j=1}^k [M_j, M_j+N_j]$. Then, we have the following bound
\[\sum_{\alpha\in A} |S(\alpha)|^2\leq \left(\prod_{j=1}^k (N_j^{1/2}+\delta_j^{-1/2})^2\right)\sum_{n\in B} |c(n)|^2.\]
\end{theorem}

\par Fix a number $N\in \Z_{\geq 1}$ and for each prime $p$, we fix a subset $\Omega(p)$ of $\left(\Z/p^N \Z \right)^k$ for each prime $p$. For each integral vector $m\in \Z^k$, real number $x>0$, we define \[P(x;m):=\#\left\{p\leq x \mid m\mod{p^N}\in \Omega(p)\right\}\] and 
\[P(x):=\sum_{p\leq x} \#\Omega(p) p^{-Nk}.\]
In the statement of the result below, $x$ will refer to a positive real number.
\begin{theorem}\label{mean square}
For each prime $p$, let $\Omega(p)$ be a subset of $\left(\Z/p^N\Z\right)^k$. Let \[B=\prod_{j=1}^k [M_j, M_j+N_j]\] be a box in $\mathbb{R}^k$ such that $N_j\geq x^{2N}$ for all $j$. Let $V(B)$ be the volume of the box $V(B)=\prod_j N_j$, then
\[\sum_{m\in B\cap \Z^k} \left(P(x;m)-P(x)\right)^2\ll_{k,N} V(B) P(x).\]
 \end{theorem}
The notation above means that there is a constant $C>0$ which depends only on $k$ and $N$, and not on $x$, such that $\sum_{m\in B\cap \Z^k} \left(P(x;m)-P(x)\right)^2\leq C V(B) P(x)$ for all real numbers $x>0$.
\begin{proof}
Let $A=\bigcup_{p\leq x}A_p$, where \[A_p=\left\{\alpha_1(p), \alpha_2(p),\dots, \alpha_{p^{Nk}-1}(p)\right\}\] is a complete set of representatives $\frac{1}{p^N} \Z^k/\Z^k$. The set is $\delta$ well-spaced for $\delta_i=\frac{1}{x^{2N}}$ for all $i$. A standard argument shows that Theorem \ref{huxleythm} can be applied to $A$, in order to deduce the statement of Theorem \ref{mean square}. This follows verbatim from the argument given in \cite[p.93, l.15 to p.94, l.10]{gallagher1973large}.
\end{proof}
We next apply the above result to studying the average distribution of local torsion primes. Let $k,N=2$ and for each prime $p$, let $\Omega(p):=\mathfrak{A}_p\subset \Z/p^2\times \Z/p^2$. Thus, for $Y>0$, we have that $P(Y)=\sum_{p\leq Y} \frac{\# \mathfrak{A}_p}{p^4}$. Let $C,D>Y^4$, and consider the set $\mathcal{S}_{C,D}$ of all pairs $m=(a,b)\in \Z\times \Z$ with $|a|<C$ and $|b|<D$. Fix a number $\beta>0$. Let $\mathcal{S}_{C,D}^\beta$ be the subset of $m=(a,b)\in\mathcal{S}_{C,D}$ such that 
\[|\#\Pi_{lt,E_{a,b}}(Y)-P(Y)|<\beta \sqrt{P(Y)},\] and set $\mathcal{Q}_{C,D}^\beta:=\mathcal{S}_{C,D}\backslash \mathcal{S}_{C,D}^\beta$.

\begin{theorem}\label{th s5 main}
Let $Y>0$ be a fixed number. Then, there is an absolute constant $c>0$ such that for any number $\beta>0$ the proportion of elliptic curves $E$ for which 
\begin{equation}\label{boring eqn}|\#\Pi_{lt,E}(Y)-P(Y)|<\beta \sqrt{P(Y)}\end{equation} is $\geq (1-\frac{c}{\beta^2})$. Furthermore, if $C,D>Y^4$, we have that $\frac{\#\mathcal{Q}_{C,D}^\beta}{\#\mathcal{S}_{C,D}}<\frac{c}{\beta^2}$ for any $\beta>0$.
\end{theorem}
\begin{remark}
Setting $\beta=10\sqrt{c}$, we find that $\geq 99\%$ of elliptic curves satisfy \eqref{boring eqn}.
\end{remark}
\begin{proof}
Let $C$ and $D$ be integers such that $C,D>Y^4$. We are to show that there is an absolute constant $c>0$ such that $\frac{\#\mathcal{Q}_{C,D}^\beta}{\#\mathcal{S}_{C,D}}<\frac{c}{\beta^2}$ for any $\beta>0$. This implies that the proportion of elliptic curves $E_{/\Q}$ (ordered by height) for which 
\[|\#\Pi_{lt,E}(Y)-P(Y)|\geq \beta \sqrt{P(Y)}\]
is $\leq c/\beta^2$.
\par Let $E_{a,b}$ be an elliptic curve associated with a minimal pair $m=(a,b)\in \Z\times \Z$. Note that by definition, $P(Y;m)=\#\Pi_{lt, E}(Y)$.
\par According to the square sieve inequality of Theorem \ref{mean square}, 
\begin{equation}\label{last th e1}\sum_{m\in \mathcal{S}_{C,D}}\left(P(m;Y)-P(Y)\right)^2<  c\#\mathcal{S}_{C,D} P(Y),\end{equation} where $c>0$ is an absolute constant. On the other hand, 
\begin{equation}\label{last th e2}\begin{split}\sum_{m\in \mathcal{S}_{C,D}}\left(P(m;Y)-P(Y)\right)^2\geq & \sum_{m\in \mathcal{Q}_{C,D}}\left(P(m;Y)-P(Y)\right)^2,\\
= & \sum_{m\in \mathcal{Q}_{C,D}}\left(\#\Pi_{lt,E_{m}}(Y)-P(Y)\right)^2,\\
\geq  & \sum_{m\in \mathcal{Q}_{C,D}}\beta^2 P(Y)= \# \mathcal{Q}_{C,D} \beta^2 P(Y).\end{split}\end{equation}
Combining the inequalities from \eqref{last th e1} and \eqref{last th e2}, we have shown that $\frac{\#\mathcal{Q}_{C,D}^\beta}{\#\mathcal{S}_{C,D}}<\frac{c}{\beta^2}$, and this completes the proof.
\end{proof}

The above result thus proves an explicit result for the expectation for $\#\Pi_{lt,E}(Y)$, which is shown to be close to $P(Y)= \sum_{p\leq Y}\frac{\# \mathfrak{A}_p}{p^4}$.
\newpage
\section{Numerical data}
\begin{center}
\begin{table}[h]
\caption{Data for $\mathfrak{S}_p/p^2$ for primes $7\leq p<150$.}
\label{tab:1}
\begin{tabular}{ |c|c|c|c| }
\hline
$p$ & $\#\mathfrak{S}_p'/p^2$ & $p$ & $\#\mathfrak{S}_p'/p^2$ \\ 
\hline
 7 & 0.0816326530612245 & 71 & 0.0208292005554453 \\
11 & 0.0413223140495868 & 73 & 0.0270219553387127 \\
13 & 0.0710059171597633 & 79 & 0.0374939913475405\\
17 & 0.0276816608996540 & 83 & 0.0178545507330527 \\
19 & 0.0581717451523546 & 89 & 0.0222194167403106 \\
23 & 0.0415879017013233 & 97 & 0.0255074928260176\\
29 & 0.0332936979785969 & 101 & 0.00980296049406921 \\
31 & 0.0312174817898023 & 103 & 0.0288434348194929 \\
37 & 0.0306793279766253 & 107 & 0.00925845051969604 \\
41 & 0.0118976799524093 & 109 & 0.0181802878545577 \\
43 & 0.0567874526771228 & 113 & 0.0263137285613595 \\
47 & 0.0208239022181983 & 127 & 0.0169260338520677 \\
53 & 0.0277678889284443 & 131 & 0.0189382903094225 \\
59 & 0.0166618787704683 & 137 & 0.0108689860940913 \\
61 & 0.0349368449341575 & 139 & 0.0142849748977796 \\
67 & 0.0147026063711294 & 149 & 0.0133327327597856 \\
 \hline
\end{tabular}
\end{table}
\end{center}
\newpage
\subsection*{Data availability statement} No data was generated or analyzed in this paper.
\bibliographystyle{alpha}
\bibliography{references}
\end{document}